\documentclass{amsart}
\usepackage{amsmath}
\usepackage{amssymb}
\usepackage{latexsym}
\usepackage{amsbsy}
\usepackage[all]{xy}
\usepackage{amscd}
\usepackage[dvips]{graphicx}
\newtheorem{theorem}{Theorem}[section]
\newtheorem{lemma}[theorem]{Lemma}

\theoremstyle{definition}

\newtheorem{Prop}[theorem]{Proposition}
\newtheorem{Cor}[theorem]{Corollary}

\theoremstyle{remark}
\newtheorem{remark}[theorem]{Remark}

\numberwithin{equation}{section}

\newcommand\bbn{{\mathbb N}}
\newcommand\bbz{{\mathbb Z}}

\newcommand\bbc{{\mathbb C}}

\newcommand\aut{\mbox{Aut}}

\newcommand\ind{\mbox{ind}\,}

\newcommand\ed{\mbox{End}}

\newcommand\nd{{\noindent}}
\newcommand\mc{{\mathcal{C}}}

\begin{document}

\title{Hall algebras associated to triangulated categories, II: almost associativity}

\author{Fan Xu}
\address{Department of Mathematics, Tsinghua University, Beijing 100084, P.R.China}

\email{fanxu@mail.tsinghua.edu.cn}
\thanks{The research was
supported in part by the Ph.D. Programs Foundation of Ministry of
Education of China (No. 200800030058)}


\date{October 28, 2007. Last modified: October 09, 2009.}

\dedicatory{Dedicated to Professor YingBo Zhang.}


\begin{abstract}
By using the approach in \cite{XX2006} to Hall algebras arising in
homologically finite triangulated categories, we find an `almost'
associative multiplication structure for indecomposable objects in a
2-periodic triangulated category. As an application, we give a new
proof of the theorem of Peng and Xiao in \cite{PX2000} which
provides a way of realizing symmetrizable Kac-Moody algebras and
elliptic Lie algebras via 2-periodic triangulated categories.
\end{abstract}

\maketitle

\section*{Introduction}
Let $\mathcal{U}$ be the universal enveloping algebra of a simple
Lie algebra of type $A, D$ or $E$ over the field of rational numbers
$\mathbb{Q}.$ There are many interesting results involving the
categorification and the geometrization of $\mathcal{U}.$ The work
of Gabriel \cite{Gabriel1972} strongly suggested the possibility of
the categorification. He showed that there exists a bijection
between isomorphism classes of all indecomposable modules over a
hereditary algebra of Dynkin type and the positive roots of the
corresponding semisimple Lie algebra. In \cite{Ringel1990}, Ringel
explicitly realized the positive part of $\mathcal{U}$ through the
Hall algebra approach. A different but somewhat parallel realization
was given by Lusztig. He showed that the negative part of
$\mathcal{U}$ can be geometrically realized by using constructible
functions on affine spaces of representations of a preprojective
algebra in \cite{Lusztig2000}. One may naturally consider to recover
the whole Lie algebras and the whole (quantized) enveloping algebras
\cite{Ringel1990}.

Nakajima \cite{Nakajima1998} showed that an arbitrarily large
finite-dimensional quotient of $\mathcal{U}$ can be realized in
terms of the homology of a triple variety. A different construction
was given by Lusztig  in \cite{Lusztig2000} in terms of
constructible functions on the triple variety. On the other hand,
Peng and Xiao \cite{PX2000} defined a Lie bracket between two
isomorphism classes of indecomposable objects in a $k$-additive
triangulated category with the translation functor $T$ satisfying
$T^2=1$ for a finite field $k$ with the cardinality $q$. It induces
a Lie algebra over $\bbz/(q-1)$, while it is still unknown which
associative multiplication induces the Lie bracket over
$\bbz/(q-1)$. However, it seems to be hopeless to realize the whole
enveloping algebra by the constructions of Nakajima
\cite{Nakajima1998} and Lusztig \cite{Lusztig2000}.

Recently, To\"en gave a multiplication formula which defines an
associative algebra (called the derived Hall algebra) corresponding
to a dg category \cite{Toen2005}. In \cite{XX2006}, we extended to
prove that To\"en's formula can be applied to define an associative
algebra for any triangulated category with some homological
finiteness conditions. Unfortunately, a 2-periodic triangulated
category does not satisfy these homological finiteness conditions in
general. Hence, To\"en's formula can not supply the realization of
quantum groups. However, the approach in \cite{XX2006} strongly
suggests the possibility to construct an associative multiplication
over $\bbz/(q-1)$.

Inspired by the method discussed in \cite{XX2006},  in this paper,
we prove that there exists an `almost' associative multiplication
over $\bbz[\frac{1}{q}]/(q-1)$ for isomorphism classes of
indecomposable objects in a 2-periodic triangulated category
(Corollary \ref{almostassociativity}). The associativity of the
multiplication heavily depends on the choice of structure constants
(Hall numbers) for defining the multiplication. The key techniques
in this paper are to substitute derived Hall numbers in
\cite{Toen2005} or \cite{XX2006} for Hall numbers in \cite{PX2000}
and introduce new variants associated to indecomposable objects. As
a direct application, we obtain a new proof of the theorem of Peng
and Xiao in \cite{PX2000}, i.e., Theorem \ref{maintheorem} in
Section 3.

\section{The `almost' associativity}
Given a finite field $k$ with $q$ elements, let $\mc_2$ be a
$k$-additive triangulated category with the translation functor
$T=[1]$ satisfying $(1)$  the homomorphism space
$\mathrm{Hom}_{\mc_2}(X,Y)$ for any two objects $X$ and $Y$ in $\mc$
is a finite dimensional $k$-space, $(2)$ the endomorphism ring $\ed
X$ for any indecomposable object $X$ is finite dimensional local
$k$-algebra and $(3)$ $T^2\cong 1$. Then the category $\mc_2$ is
called a 2-periodic triangulated category. For any $M\in \mc_2$, we
set $\mathrm{\underline{dim}}M$ to be the canonical image of $M$ in
the Grothendieck group of $\mc_2$. Throughout this paper, we assume
$\mc_2$ is proper, i.e., for any nonzero indecomposable object $X$
in $\mc_2$, $\mathrm{\underline{dim}}X\neq 0.$ By $\mbox{ind}\mc_2$
we denote the set of representatives of isomorphism classes of all
indecomposable objects in $\mc_2.$ For any indecomposable object
$X\in \mc_2$, we set
$d(X)=\mathrm{dim}_{k}(\mathrm{End}X/\mathrm{radEnd}X).$ Recall that
(see \cite[Lemma 8.1]{PX2000})
$$
|\mathrm{Aut}X|=|\mathrm{radEnd}X|(q^{d(X)}-1).
$$ For any $X, Y$ and $Z$ in $\mc_2$, we will use $fg$ to denote the
composition of morphisms $f:X\rightarrow Y$ and $g:Y\rightarrow Z,$
and $|A|$ to denote the cardinality of a finite set $A.$ Given
$X,Y;L\in \mc_2,$ put
$$ W(X,Y;L)=\{(f,g,h)\in \hom(X,L)\times
\hom(L,Y)\times(Y,X[1])\mid$$$$
X\xrightarrow{f}L\xrightarrow{g}Y\xrightarrow{h}X[1] \mbox{ is a
triangle}\}.$$ \\
There is a natural action of $\aut X\times \aut Y$ on $W(X,Y;L).$
The orbit space is denoted by $V(X,Y;L).$ The orbit of $(f,g,h)\in
W(X,Y;L)$ is denoted by $(f,g,h)^{\wedge}$. Then
$$(f,g,h)^{\wedge}=\{(af,gc^{-1},ch(a[1])^{-1})\mid (a,c)\in \aut X\times \aut Y\}.$$
We also write $F_{XY}^{L}=|V(X,Y;L)|.$  Throughout this section, we
fix a triple pair $(X, Y, Z)$ such that $X, Y, Z$ are nonzero
indecomposable objects in $\mc_2$ and none of the following
conditions holds
\begin{enumerate}
    \item $ X\cong Z\cong Y[1]$,
    \item $ X\cong Y\cong Z[1]$,
    \item $ Y\cong Z\cong X[1]$.
\end{enumerate}

\begin{lemma}\label{mainlemma}
Let $X,Y,Z$ and $M\neq 0$ be in $\mc_2$ with $X,Y, Z$
indecomposable and $M\oplus Z\ncong X\oplus Y.$ Then we have
$$\frac{F_{XY}^{M\oplus
Z}}{|\mathrm{Aut}Z|}\in \bbz[\frac{1}{q}].$$ where
$\bbz[\frac{1}{q}]$ is the polynomial ring for $\frac{1}{q}$ with
coefficients in $\bbz.$
\end{lemma}
\begin{proof}
We define the action of $\aut Z$ on $V(X,Y;M\oplus Z)$ as follows.
For any $\alpha=(\left(%
\begin{array}{cc}
  f_1 & f_2 \\
\end{array}%
\right),\left(%
\begin{array}{c}
  g_1 \\
  g_2 \\
\end{array}%
\right), h)^{\wedge}\in V(X,Y; M\oplus Z)$ and $d\in \aut Z,$ define
$$
d.(\left(%
\begin{array}{cc}
  f_1 & f_2 \\
\end{array}%
\right),\left(%
\begin{array}{c}
  g_1 \\
  g_2 \\
\end{array}%
\right), h)^{\wedge}=(\left(%
\begin{array}{cc}
  f_1 & f_2d^{-1} \\
\end{array}%
\right),\left(%
\begin{array}{c}
  g_1 \\
  dg_2 \\
\end{array}%
\right), h)^{\wedge}.
$$
The orbit space is denoted by $\overline{V}(X,Y; M\oplus Z).$ Let
$$
G_{\overline{\alpha}}=\{d\in \mathrm{Aut}Z\mid (\left(%
\begin{array}{cc}
  f_1 & f_2d^{-1} \\
\end{array}%
\right),\left(%
\begin{array}{c}
  g_1 \\
  dg_2 \\
\end{array}%
\right), h)^{\wedge}=(\left(%
\begin{array}{cc}
  f_1 & f_2 \\
\end{array}%
\right),\left(%
\begin{array}{c}
  g_1 \\
  g_2 \\
\end{array}%
\right), h)^{\wedge}\}.
$$
 Then by definition,
$G_{\overline{\alpha}}$ is equal to
$$
\{d\in \mathrm{Aut}Z\mid(\left(%
\begin{array}{cc}
  f_1 & f_2d^{-1} \\
\end{array}%
\right),\left(%
\begin{array}{c}
  g_1 \\
  dg_2 \\
\end{array}%
\right), h)=(\left(%
\begin{array}{cc}
  af_1 & af_2 \\
\end{array}%
\right),\left(%
\begin{array}{c}
  g_1b^{-1} \\
  g_2b^{-1} \\
\end{array}%
\right), bh(a[1])^{-1})
$$
$$
\mbox{ for some }(a,b)\in \mathrm{Aut}X\times \mathrm{Aut}Y\},
$$ or equivalently, equal to $$\{d\in \mathrm{Aut}Z\mid dg_2=g_2b, g_1b=g_1, af_1=f_1\mbox{ and } f_2d=af_2$$$$ \mbox{ for some }(a,b)\in \mathrm{Aut}X\times \mathrm{Aut}Y\}. $$Given $(b, d)\in \mathrm{Aut}Y\times \mathrm{Aut}Z$ such that
$dg_2=g_2b$ and $g_1b=g_1$,  we have the following diagram with the
middle square being commutative
$$
\xymatrix{X\ar[rr]^-{\left(%
\begin{array}{cc}
  f_1 & f_2 \\
\end{array}%
\right)}&&M\oplus Z\ar[rr]^{\left(%
\begin{array}{c}
  g_1 \\
  g_2 \\
\end{array}%
\right)}&& Y\ar[rr]^h&& X[1]\\
X\ar[rr]_-{\left(%
\begin{array}{cc}
  f_1 & f_2 \\
\end{array}%
\right)}&&M\oplus Z\ar[u]_{\left(
                            \begin{array}{cc}
                              1 & 0 \\
                              0 & d \\
                            \end{array}
                          \right)
}\ar[rr]_-{\left(%
\begin{array}{c}
  g_1 \\
  g_2 \\
\end{array}%
\right)}&& Y\ar[rr]_{h}\ar[u]_{b}&& X[1]}
$$
By the axioms of triangulated categories, there exists $a\in
\mathrm{Aut} X$ such that $af_1=f_1$ and $f_2d=af_2.$ Hence,
$$G_{\overline{\alpha}}=\{d\in \mathrm{Aut}Z\mid dg_2=g_2b \mbox{ and } g_1b=g_1 \mbox{ for some }b\in \mathrm{Aut}Y \}.$$
 The map $g_1$ naturally induces a triangle
$$
\xymatrix{M\ar[r]^{g_1}&Y\ar[r]^-{h_1}& C(g_1)\ar[r]&M[1]}
$$
where $C(g_1)$ is the cone of the map $g_1.$ Let $b'=1-b.$ Then we
have
$$\{b'\mid b\in \mathrm{Aut}Y\mid g_1b=g_1\}=\{b'\in \mathrm{End}Y\mid b'\in h_1\mathrm{Hom}(C(g_1), Y), 1-b'\in \mathrm{Aut}Y\}.$$
We claim that for any $t\in \mathrm{Hom}(C(g_1), Y),$ both $th_1$
and $h_1t$ are nilpotent. Assume that $h_1t$ is not nilpotent, then
it is invertible since $Y$ is indecomposable. This implies $g_1=0.$
Consider the following diagram
$$
\xymatrix{M\ar[rr] \ar@{.>}[d]_{u}&& M\ar[rr]\ar[d]_{v}&& 0\ar[d]\ar[rr]&&M[1]\\X\ar[rr]^-{\left(%
\begin{array}{cc}
  f_1 & f_2 \\
\end{array}%
\right)}&&M\oplus Z\ar[rr]^{\left(%
\begin{array}{c}
  0 \\
  g_2 \\
\end{array}%
\right)}&& Y\ar[rr]^h&& X[1]}
$$
where $v=\left(
           \begin{array}{cc}
             1 & 0 \\
           \end{array}
         \right)
$. The middle square is commutative, then there exists $u:
M\rightarrow X$ such that the diagram is commutative. This implies
$uf_1=1_M$. However, $X$ is indecomposable and then $u$ is an
isomorphism. The morphism $g_2$ is also an isomorphism by the
octahedral axiom as showed in the following diagram.
$$
\xymatrix{&&Z\ar[d]\ar@{=}[rr]&&Z\ar[d]^{g_2}\\X\ar@{=}[d]\ar[rr]^-{\left(%
\begin{array}{cc}
  f_1 & f_2 \\
\end{array}%
\right)}&&M\oplus Z\ar[d]\ar[rr]^{\left(%
\begin{array}{c}
  0 \\
  g_2 \\
\end{array}%
\right)}&& Y\ar[d]\ar[rr]^h&&
X[1]\\X\ar[rr]^{f_1}&&M\ar[rr]&&C(g_2)\ar[rr]&&X[1]}
$$
where $C(g_2)$ is the cone of $g_2.$ It contradicts to the
assumption $M\oplus Z\cong X\oplus Y.$ Hence, $g_1\neq 0$. In the
same way, we have $g_2\neq 0.$ This implies that for any $t\in
\mathrm{Hom}(C(g_1), Y),$ both $th_1$ and $h_1t$ are nilpotent. We
have $1-h_1t\in \mathrm{Aut}Y.$ Therefore, $G_{\overline{\alpha}}$
is isomorphic to
$$G'_{\overline{\alpha}}=\{d'\in \mathrm{End}Z\mid  1-d'\in \mathrm{Aut}Z, d'g_2=g_2b' \mbox{ for
some }b'\in h_1\mathrm{Hom}(C(g_1),Y)\}.$$ Let $d'\in \mathrm{End}Z$
satisfy $d'g_2=g_2b' \mbox{ for some }b'\in h_1\hom(C(g_1),Y)\}$.
Since $b'$ is nilpotent, we assume $(b')^k=0$ for some $k\in \bbn$,
then $(d')^kg_2=0$. However, $Z$ is indecomposable and $g_2\neq 0$.
We deduce that $d'$ is nilpotent and then $1-d'\in \mathrm{Aut}Z.$
Hence, $G'_{\overline{\alpha}}$ is a vector space. Finally, we
obtain
$$
\frac{F_{XY}^{M\oplus
Z}}{|\mathrm{Aut}Z|}=\sum_{\overline{\alpha}\in \overline{V}(X,Y;
M\oplus Z)}\frac{1}{|G_{\overline{\alpha}}|}\in \bbz[\frac{1}{q}].
$$
\end{proof}
Note that the conclusion of the lemma may not hold if $M\oplus
Z\cong X\oplus Y$ in Lemma \ref{mainlemma}, then . Indeed, if
$X\ncong Y$, $F^{X\oplus Y}_{XY}=|\mathrm{Hom}(X, Y)|=q^n$ where
$n=\mathrm{dim}_{k}\mathrm{Hom}(X, Y).$ If $X\cong Y$, $F^{X\oplus
Y}_{XY}=|\mathrm{End}X|+|\mathrm{radEnd}X|.$

Denote by $(X,Y)_Z$ the subset of $\mathrm{Hom}_{\mc_2}(X,Y)$
consisting of the morphisms whose cones are isomorphic to $Z.$ The
following proposition (\cite[Proposition 2.5]{XX2006}) also holds
for 2-periodic triangulated categories.
\begin{Prop}\label{mainproposition}
For any $Z,L,M\in \mc_2,$ we have
\begin{enumerate}
    \item Any $\alpha=(l,m,n)^{\wedge}\in V(Z,L;M)$ has the representative
of the form:
\begin{equation}\nonumber
\xymatrix{Z\ar[rr]^{\left(%
\begin{array}{c}
  0 \\
  l_2 \\
\end{array}%
\right)}&& M\ar[rr]^{\left(
                       \begin{array}{cc}
                         0 & m_2 \\
                       \end{array}
                     \right)
}&& L\ar[rr]^{\left(%
\begin{array}{cc}
  n_{11} & 0 \\
  0 & n_{22} \\
\end{array}%
\right)}&& Z[1]}
\end{equation}
where $Z=Z_1(\alpha)\oplus Z_2(\alpha),$ $L=L_1(\alpha)\oplus
L_2(\alpha)$, $n_{11}$ is an isomorphism between $L_1(\alpha)$ and
$Z_1(\alpha)[1]$ and $n_{22}\in
\mathrm{radHom}(L_2(\alpha),Z_2(\alpha)[1]).$
    \item $$
\frac{|(M,L)_{Z[1]}|}{|\mathrm{Aut}L|}=\sum_{\alpha\in
V(Z,L;M)}\frac{|\mathrm{End}
L_1(\alpha)|}{|n(\alpha)\mathrm{Hom}(Z[1],L)||\mathrm{Aut}L_1(\alpha)|}
$$
and
$$
\frac{|(Z,M)_{L}|}{|\mathrm{Aut}Z|}=\sum_{\alpha\in
V(Z,L;M)}\frac{|\mathrm{End}
Z_1(\alpha)|}{|\mathrm{Hom}(Z[1],L)n(\alpha)||\mathrm{Aut}Z_1(\alpha)|}
$$
where $n(\alpha)=\left(%
\begin{array}{cc}
  n_{11} & 0 \\
  0 & n_{22} \\
\end{array}%
\right)$.
\end{enumerate}
\end{Prop}
For $\alpha\in (l,m,n)^{\wedge},$  define
$$
s(\alpha)=\mathrm{dim}_{k}n\mathrm{Hom}(Z[1],L),\quad
t(\alpha)=\mathrm{dim}_{k}\mathrm{Hom}(Z[1],L)n.
$$

For any objects $U, V$ and $W$ in $\mc_2,$ define
\begin{enumerate}
    \item $g_{UV}^{W}:=\frac{|(U,W)_V|}{|\aut
    U|}$, if $U\ncong W\oplus V[1]$,
    \item $
\bar{g}_{UV}^{W}:=\frac{|(W,V)_{U[1]}|}{|\aut V|}$, if $V\ncong
W\oplus U[1]$,
    \item $
g_{W\oplus V[1], V}^{W}:=|\mathrm{Hom}(W, V[1])|\cdot\frac{|(W\oplus
V[1],W)_V|}{|\aut (W\oplus V[1])|}=\frac{1}{|\aut V|}$,
    \item $\bar{g}_{U, W\oplus
U[1]}^{W}:=|\mathrm{Hom}(U[1], W)|\cdot\frac{|(W, W\oplus
U[1])_{U[1]}|}{|\aut(W\oplus U[1])|}=\frac{1}{|\aut U|}.$
  \end{enumerate}

Different from \cite{PX2000}, we will consider the image of
numbers in $\bbz[\frac{1}{q}]/(q-1)$ instead of $\bbz/(q-1)$ where
$\bbz[\frac{1}{q}]$ is the polynomial ring for $\frac{1}{q}$ with
coefficient in $\bbz$.

We have the following corollary of Proposition
\ref{mainproposition}.
\begin{Prop}\label{numbers}
Let $X,Y,Z, L, L'$ and $M$ be in $\mc_2$ with $X,Y, Z$ and $M\neq 0$
being indecomposable. Then  we have the following properties.

\begin{enumerate}
        \item  If $L\ncong M\oplus Z[1]$ and $L\neq 0,$ then the number
        $g_{XY}^{L}g_{ZL}^M$ belongs to $\bbz[\frac{1}{q}]$.
        \item  If $L\cong M\oplus Z[1]$ and $L\ncong X\oplus Y,$ then the number
        $g_{XY}^{L}g_{ZL}^M$ belongs to $\bbz[\frac{1}{q}]$.
        \item  If $L'\ncong M\oplus Y[1]$ and $L'\neq 0,$ then the number
        $g_{ZX}^{L'}g_{L'Y}^M$ belongs to $\bbz[\frac{1}{q}]$.
        \item  If $L'\cong M\oplus Y[1]$ and $L'\ncong X\oplus Z,$
        then $g_{ZX}^{L'}g_{L'Y}^M$ belongs to $\bbz[\frac{1}{q}]$.
        \item If $X\ncong Y$ and $X\ncong Y[1],$ then the numbers $g_{XY}^{X\oplus Y}g_{X[1],X\oplus Y}^Y-g_{X[1], X}^{0}g_{0, Y}^Y$ and $g_{XY}^{X\oplus Y}g^{X}_{Y[1], X\oplus Y}-g_{Y[1], X}^{X\oplus Y[1]}g_{X\oplus Y[1], Y}^X$ belong to
                $\bbz[\frac{1}{q}]$.
        \item If $Z\ncong X$ and $Z\ncong X[1],$ then the number $g_{ZX}^{Z\oplus X}g_{Z\oplus
             X, X[1]}^Z-g_{X, X[1]}^{0}g_{Z,0}^Z$ and $g_{ZX}^{Z\oplus X}g^{X}_{Z\oplus
           X, Z[1]}-g_{X, Z[1]}^{X\oplus Z[1]}g_{Z, X\oplus
              Z[1]}^X$ belong to
                $\bbz[\frac{1}{q}]$.
        \item   $g_{XY}^{L}g_{ZL}^M-g_{XY}^{L}\bar{g}_{ZL}^M\in \bbz[\frac{1}{q}].$ In
        $\bbz[\frac{1}{q}]/(q-1),$ $g_{XY}^{L}g_{ZL}^M-g_{XY}^{L}\bar{g}_{ZL}^M=0.$
        \item   $g_{ZX}^{L'}g_{L'Y}^M-\bar{g}_{ZX}^{L'}g_{L'Y}^M\in \bbz[\frac{1}{q}].$ In $\in \bbz[\frac{1}{q}]/(q-1)$, $g_{ZX}^{L'}g_{L'Y}^M-\bar{g}_{ZX}^{L'}g_{L'Y}^M=0.$

\end{enumerate}
\end{Prop}
\begin{proof}
If $L\ncong M\oplus Z[1],$ then by Proposition
\ref{mainproposition}, we have $L_1(\alpha)=0$ and
$$g_{ZL}^{M}=\sum_{\alpha\in
V(Z,L;M)}\frac{1}{|n(\alpha)\mathrm{Hom}(Z[1],L)|}\in
\bbz[\frac{1}{q}].$$ In the same way, $g_{XY}^{L}\in
\bbz[\frac{1}{q}].$ This proves (1). Next, we prove (2). In this
case, $g_{ZL}^{M}=\frac{1}{|\mathrm{Aut} Z|}$. As in the proof of
Lemma \ref{mainlemma}, $g_{XY}^{M\oplus Z[1]}$ is equal to
$$\sum_{\alpha\in V(X, Y; M\oplus Z[1])}\frac{1}{|h(\alpha)\mathrm{Hom}(X[1], Y)|}=\sum_{\overline{\alpha}\in \overline{V}(X, Y; M\oplus Z[1])}\frac{1}{|h(\alpha)\mathrm{Hom}(X[1], Y)|}\frac{|\mathrm{Aut}Z|}{|G_{\overline{\alpha}}|}.$$
This proves (2). The proofs of (3) and (4) are similar. As for (5),
the number $g_{XY}^{X\oplus Y}g_{X[1],X\oplus Y}^Y-g_{X[1],
X}^{0}g_{0, Y}^Y$ is equal to $$\frac{|\mathrm{Hom}(X,
Y)|-1}{|\mathrm{Aut}
X|}=\frac{q^{\mathrm{dim}_{k}\mathrm{Hom}(X,Y)}-1}{|\mathrm{radEnd}
X|(q^{d(X)}-1)}.$$ Since
$\frac{\mathrm{dim}_{k}\mathrm{Hom}(X,Y)}{d(X)}=l_{End
X}(\mathrm{Hom}(X,Y))\in \bbz,$ the number belongs to
$\bbz[\frac{1}{q}]$. The proofs of the rest part of (5) and (6) are
similar. We prove (7). If $L=0$, then
$g_{XY}^{L}g_{ZL}^{M}-g_{XY}^{L}\bar{g}_{ZL}^M=\frac{1}{|\mathrm{Aut}X|}-\frac{1}{|\mathrm{Aut}X|}=0$.
If $L\cong M\oplus Z[1],$ then
$g_{ZL}^M=\bar{g}_{ZL}^M=\frac{1}{|\mathrm{Aut} Z|}.$ Hence,
$g_{XY}^{L}g_{ZL}^{M}-g_{XY}^{L}\bar{g}_{ZL}^M=0$. If $L\ncong
M\oplus Z[1]$ and $L\neq 0,$ then by Proposition
\ref{mainproposition}, we have $g_{XY}^L\in \bbz[\frac{1}{q}]$ and
$g_{ZL}^M-\bar{g}_{ZL}^M=\sum_{\alpha\in
V(Z,L;M)}(\frac{1}{q^{s(\alpha)}}-\frac{1}{q^{t(\alpha)}})\in
\bbz[\frac{1}{q}]$. This proves (7). The proof of (8) is similar.
\end{proof}

\begin{Cor}\label{transitive}
Let $X,Y,Z$ and $M$ be in $\mc_2$ with $X,Y, Z$ and $M\neq 0$ being
indecomposable. Then we have
$$
\sum_{[L], L\in
\mc_2}g_{XY}^{L}g_{ZL}^{M}-\sum_{[L'],L'\in\mc_2}g_{ZX}^{L'}g_{L'Y}^{M}\in
\bbz[\frac{1}{q}]$$ and in $\bbz[\frac{1}{q}]/(q-1),$
$$
 \sum_{[L], L\in
\mc_2}g_{XY}^{L}g_{ZL}^{M}-\sum_{[L'],L'\in\mc_2}g_{ZX}^{L'}g_{L'Y}^{M}
   =\sum_{[L], L\in
\mc_2}g_{XY}^{L}\bar{g}_{ZL}^{M}-\sum_{[L'],
L'\in\mc_2}\bar{g}_{ZX}^{L'}g_{L'Y}^{M}.
$$
\end{Cor}
\begin{proof}
 As for the first statement of the corollary,
by Proposition \ref{numbers} (1) and (2), it is sufficient to prove
that the number
$$g_{XY}^{X\oplus Y}g_{Z, X\oplus
Y}^M+\delta_{X,Y[1]}g_{XY}^0g_{Z,0}^M-g_{ZX}^{Z\oplus X}g_{Z\oplus
X, Y}^M-\delta_{X,Z[1]}g_{ZX}^0g_{0,Y}^M\in \bbz[\frac{1}{q}].$$ If
$X\oplus Y\ncong M\oplus Z[1]$ and $Z\oplus X\ncong M\oplus Y[1],$
by Proposition \ref{numbers} (1) and (3), this number belongs to
$\bbz[\frac{1}{q}]$. Hence, we only need to check the following
cases:
\begin{enumerate}
    \item $M\cong X, Y\cong Z[1], X\ncong Y, X\ncong Y[1]$,
    \item $M\cong Y, X\cong Z[1], Y\ncong X, Y\ncong X[1]$,
    \item $M\cong Z, X\cong Y[1], Z\ncong X, Z\ncong X[1]$.
\end{enumerate}
We note that the cases (1),(2) and (3) are symmetric to each other.
Proposition \ref{numbers} (5) and (6) correspond to the cases
(1),(2) and (3), respectively.  This proves the first statement of
the corollary. The second statement of the corollary can be deduced
by the first statement  and Proposition \ref{numbers} (7) (8).
\end{proof}
Here we recall some notations in \cite{XX2006}. Let $X,Y,Z, L, L'$
and $M$ be in $\mc_2$. Then define
$$
\mathrm{Hom}(M\oplus X,L)^{Y,Z[1]}_{L'[1]}:=\{\left(%
\begin{array}{c}
  m \\
  f \\
\end{array}%
\right)\in \mathrm{Hom}(M\oplus X,L)\mid$$$$\hspace{3cm}
\mathrm{Cone}(f)\simeq Y, \mathrm{Cone}(m)\simeq Z[1] \mbox{ and }\mathrm{Cone}\left(%
\begin{array}{c}
  m \\
  f \\
\end{array}%
\right)\simeq L'[1]\}
$$
and
$$
\mathrm{Hom}(L',M\oplus X)^{Y,Z[1]}_{L}:=\{(f',-m')\in
\mathrm{Hom}(L',M\oplus X)\mid
$$$$\hspace{3cm}\mathrm{Cone}(f')\simeq Y, \mathrm{Cone}(m')\simeq Z[1]
\mbox{ and } \mathrm{Cone}(f',-m')\simeq L\}.
$$
The orbit space of $\mathrm{Hom}(M\oplus X,L)^{Y,Z[1]}_{L'[1]}$
under the action of $\aut L$ is just the orbit space of
$\mathrm{Hom}(L',M\oplus X)^{Y,Z[1]}_{L}$  under the action of $\aut
L'$ (see \cite{XX2006}).  We denoted by $V(L',L;M\oplus X)_{Y,Z[1]}$
the orbit space. The following diagram illustrates the relation
among $V(X, Y; L),$$ V(Z, L; M),$$ V(Z, X; L')$, $V(L', Y; M)$ and
$V(L',L;M\oplus X)_{Y,Z[1]}$.
\begin{equation}
\xymatrix{Z\ar@{=}[r]\ar[d]^{l'} &Z\ar[d]^l\\
L'\ar@{.>}[r]^{f'}\ar[d]^{m'} &M\ar@{.>}[r]^{g'}\ar[d]^m &Y\ar@{.>}[r]^-{h'}\ar@{=}[d] &L'[1]\ar[d]^{m'[1]}\\
X\ar[r]^f\ar[d]^{n'} &L\ar[r]^g\ar[d]^n &Y\ar[r]^-{h} &X[1]\\
Z[1]\ar@{=}[r] &Z[1]}
\end{equation}

\begin{lemma}\label{commonsummand}
Let $\alpha\in V(L',L;M\oplus X)_{Y,Z[1]}$ has the representative of
the form as follows:
$$\xymatrix{L'_1\oplus L'_2\ar[rr]^-{\left(%
\begin{array}{cc}
  0  & 0 \\
  f' & -m' \\
\end{array}%
\right)}&&M\oplus X\ar[rr]^-{\left(%
\begin{array}{cc}
  0&m \\
  0&f \\
\end{array}%
\right)}&&L_1\oplus L_2\ar[rr]^-{\left(\begin{array}{cc}
  \theta_1&0\\
  0& \theta_2 \\
\end{array}\right)}&&L'_1[1]\oplus L'_2[1]}$$
such that $\theta_1$ is a nonzero isomorphism where $L=L_1\oplus
L_2$ and $L'=L'_1\oplus L'_2.$ Then we have
$$
M\cong X,\quad Z\cong Y[1],\quad L\cong X\oplus Y,\quad L'\cong
X\oplus Y[1].
$$
\end{lemma}
\begin{proof}
By definition, the triangle $\alpha$ induces the triangles
$$
\xymatrix{X\ar[rr]^-{\left(%
\begin{array}{cc}
  0 & f\\
\end{array}%
\right)}&& L_1\oplus L_2\ar[rr]^-{\left(%
\begin{array}{c}
  g_1 \\
  g_2 \\
\end{array}%
\right)}&& Y\ar[rr]&&X[1]}
$$
and
$$
\xymatrix{L'_1\oplus L'_2\ar[rr]^-{\left(%
\begin{array}{c}
  0 \\
  f'
\end{array}%
\right)}&& M\ar[rr]&& Y\ar[rr]&&L'_1[1]\oplus L'_2[1]}.
$$
Consider the following diagram
$$
\xymatrix{X\ar[d]\ar[rr]^-{\left(%
\begin{array}{cc}
  0 & f\\
\end{array}%
\right)}&& L_1\oplus L_2\ar[d]\ar[rr]^-{\left(%
\begin{array}{c}
  g_1 \\
  g_2 \\
\end{array}%
\right)}&& Y\ar@{.>}[d]^{v}\ar[rr]&&X[1]\\
0\ar[rr]&&L_1\ar[rr]^{1}&&L_1\ar[rr]&&0}.
$$
The left square is commutative. Then there exists a map $v$ such
that $g_2v=1$.  Since $Y$ is indecomposable, we deduce $Y\cong L_1$
and $X\cong L_2$. In the same way, we deduce $L'_1\cong
Y[1],L'_2\cong M.$ If we consider the other two induced triangles,
we obtain $L_1\cong Z[1],L_2\cong M$ and $L'_1\cong Z, L'_2\cong X.$
\end{proof}

We define the action of $\aut X$ on $V(L',L;M\oplus X)_{Y,Z[1]}$ by
$$a.((f',-m'),\left(%
\begin{array}{c}
  m \\
  f \\
\end{array}%
\right),\theta))^{\wedge}=((f',-m'a^{-1}),\left(%
\begin{array}{c}
  m \\
  af \\
\end{array}%
\right),\theta))^{\wedge}$$ for any $a\in \aut X$ and $((f',-m'),\left(%
\begin{array}{c}
  m \\
  f \\
\end{array}%
\right),\theta))^{\wedge}\in V(L',L';M\oplus X)_{Y,Z[1]}.$ The orbit
space is denoted by $\bar{V}(L',L';M\oplus X)_{Y,Z[1]}.$
\begin{lemma}\label{stablesubgroup}
For any $\alpha\in V(L',L;M\oplus X)_{Y,Z[1]},$ the stable subgroup
$G_{\alpha}$ of $\alpha$ under the action of $\aut X$ is isomorphic
to a vector space if $L\neq 0$ and $L'\neq 0.$
\end{lemma}
\begin{proof}
By definition, $$G_{\alpha}=\{a\in \aut X\mid af=fb \mbox{ for
some }b\in \aut L \mbox{ such that }mb=m\}.$$

\nd (1) If $f=0,$ then $Y\cong L\oplus X[1].$ However, $Y$ is
indecomposable, so $L=0.$ Hence, $Y\cong X[1]$ and $Z\cong M.$ This
imply $L'\cong Z\oplus X.$ In the following, we assume $f\neq 0.$

\nd (2) If $L\ncong M\oplus Z[1],$ then $mb=b$ implies $b=1-b'$
with $b'\in n\hom(Z[1],L)$ nilpotent. In this case, $G_{\alpha}$
is isomorphic to
$$\hspace{1.3cm}G'_{\alpha}=\{a'\in \mathrm{End}X\mid 1-a'\in \aut X, a'f=fb' \mbox{ for
some }b'\in n\hom(Z[1],L)\}.$$ Since $b'$ is nilpotent, $a'$ is
nilpotent. And since $X$ is indecomposable, $1-a'\in \aut X.$ We
claim that $G'_{\alpha}$ is a vector space. Indeed, $0\in
G'_{\alpha}.$ For any $a'_1,a'_2\in G'_{\alpha},$ we have
$a'_if=fb'_i$ for $i=1,2$ and some $b'_1,b'_2\in n\hom(Z[1],L).$ So
$(a'_1+a'_2)f=f(b'_1+b'_2).$ It is clear that $b'_1+b'_2\in
n\hom(Z[1],L)$ and $1-(a'_1+a'_2)\in \aut X$ since $a'_1+a'_2$ is
nilpotent. This shows $a'_1+a'_2\in G'_{\alpha}.$

\nd (3) If $L\cong M\oplus Z[1],$ $\alpha$ has the
    representative of the following form:
$$\xymatrix{L'\ar[rr]^-{\left(%
\begin{array}{cc}
  f' & -m' \\
\end{array}%
\right)}&&M\oplus X\ar[rr]^-{\left(%
\begin{array}{cc}
  1&0 \\
  f_1&f_2 \\
\end{array}%
\right)}&&M\oplus Z[1]\ar[rr]^-{\theta}&&L'[1]}.$$ The stable
subgroup is
$$G_{\alpha}=\{a\in \aut X\mid af=fb \mbox{ for
some }b=\left(%
\begin{array}{cc}
  1 & 0 \\
  b_{21} & b_{22} \\
\end{array}%
\right)\in \aut L \}$$ where $f=(f_1,f_2).$ Equivalently, we have
$$
af_1=f_1+f_2b_{21} \mbox{ and }af_2=f_2b_{22}.
$$
Set $a'=a-1$ and $b'_{22}=b_{22}-1.$ If $a'\in \aut X,$ then
$$
a'\left(%
\begin{array}{cc}
  f_1 & f_2 \\
\end{array}%
\right)\left(%
\begin{array}{cc}
  1 & 0 \\
  -(b'_{22})^{-1}b_{21} & 1 \\
\end{array}%
\right)=\left(%
\begin{array}{cc}
  0 & a'f_2 \\
\end{array}%
\right).
$$
Since $Y$ is indecomposable, this shows $X\cong Z[1]$ and $Y\cong
M.$ In this case, $L'=0.$
\end{proof}
For any $X\in \mc_2$, we denote by $u_X$  by the isomorphism class
of $X$. Define the multiplication by $$u_X u_Y
:=\sum_{[L]}g_{YX}^{L}u_{L}
$$
and we set $u_X u_Y(L)=g_{XY}^L.$
\begin{theorem}\label{associativity}
Let $X,Y,Z$ and $M\neq 0$ be indecomposable objects in $\mc_2$. Then
in $\bbz[\frac{1}{q}]/(q-1)$ we have
\begin{enumerate}
  \item If $X\cong Y[1],
X\ncong Z$, $X\ncong Z[1]$ and $M\cong Z,$ then
$$[(u_Yu_X)u_Z-u_Y(u_Xu_Z)](M)=-\frac{\mathrm{dim}_{k}\mathrm{Hom}(Z,
X)}{d(X)}.$$
  \item If $X\cong Z[1], X\ncong Y, X\ncong Y[1]$ and $M\cong Y$, then
$$
[(u_Yu_X)u_Z-u_Y(u_Xu_Z)](M)=\frac{\mathrm{dim}_{k}\mathrm{Hom}(X,
Y)}{d(X)}. $$
  \item Otherwise, $$
[(u_Yu_X)u_Z-u_Y(u_Xu_Z)](M)=0.
$$
\end{enumerate}

\end{theorem}
\begin{proof}
It is equivalent to computer the number $ \sum_{[L], L\in
\mc_2}g_{XY}^{L}g_{ZL}^{M}-\sum_{[L'],L'\in\mc_2}g_{ZX}^{L'}g_{L'Y}^{M}.
$
By Corollary \ref{transitive}, we have
\begin{eqnarray}
   && \sum_{[L], L\in
\mc_2}g_{XY}^{L}g_{ZL}^{M}-\sum_{[L'],L'\in\mc_2}g_{ZX}^{L'}g_{L'Y}^{M} \nonumber\\
   &=& \sum_{[L], L\in
\mc_2}g_{XY}^{L}\bar{g}_{ZL}^{M}-\sum_{[L'], L'\in\mc_2}\bar{g}_{ZX}^{L'}g_{L'Y}^{M} \nonumber\\
  &=& \sum_{[L],L\ncong M\oplus Z[1]\in \mc_2}\frac{|(X,L)_Y|}{|\mathrm{Aut}X|}\cdot
\frac{|(M,L)_{Z[1]}|}{|\mathrm{Aut}L|}\nonumber \\
&& -\sum_{[L'],L'\ncong M\oplus
Y[1]\in\mc_2}\frac{|(L',X)_{Z[1]}|}{|\mathrm{Aut}X|}\cdot
\frac{|(L',M)_Y|}{|\mathrm{Aut}L'|}\nonumber \\
    && +|\mathrm{Hom}(Z[1], M)|\cdot\frac{|(X, M\oplus Z[1])_Y|}{|\mathrm{Aut}X|}\cdot
\frac{|(M,M\oplus Z[1])_{Z[1]}|}{|\mathrm{Aut}(M\oplus
Z[1])|}\nonumber \\
&& -|\mathrm{Hom}(M,Y[1])|\cdot\frac{|(M\oplus
Y[1],X)_{Z[1]}|}{|\mathrm{Aut}X|}\cdot \frac{|(M\oplus
Y[1],M)_Y|}{|\mathrm{Aut}(M\oplus Y[1])|}\nonumber
\\
  &=& \sum_{[L],[L'];L,L'\in\mc_2}\frac{1}{|\mathrm{Aut}X|}\cdot \delta_{L, M\oplus Z[1]}\cdot\frac{|\mathrm{Hom}(M\oplus
X,L)^{Y,Z[1]}_{L'[1]}|}{|\mathrm{Aut}L|}\nonumber
\\
&&-\sum_{[L],[L'];L,L'\in\mc_2}\frac{1}{|\mathrm{Aut}X|}\cdot
\delta'_{L', M\oplus Y[1]}\cdot\frac{|\mathrm{Hom}(L',M\oplus
X)^{Y,Z[1]}_{L}|}{|\mathrm{Aut}L'|} \nonumber
\end{eqnarray}
where $\delta_{L, M\oplus Z[1]}=|\mathrm{Hom}(Z[1], M)|$ for $L\cong
M\oplus Z[1]$ and $1$ otherwise, $\delta_{L', M\oplus
Y[1]}=|\mathrm{Hom}(M, Y[1])|$ for $L'\cong M\oplus Y[1]$ and $1$
otherwise. By Proposition \ref{mainproposition} and Lemma
\ref{commonsummand}, unless $L\cong X\oplus Y$ and $L'\cong X\oplus
Y[1],$  the above sum is equal to
$$
\sum_{[L],[L'];L,L'\in\mc_2}(\sum_{\alpha\in V(L',L;M\oplus X)_{Y,
Z[1]}}(\frac{1}{q^{s(\alpha)}}-\frac{1}{q^{t(\alpha)}}))\cdot
\frac{1}{|\mathrm{Aut}X|}
$$
where $s(\alpha)=\mathrm{dim}_{k}
\theta\mathrm{Hom}(L'[1],L)-\delta_{L, M\oplus
Z[1]}\mathrm{dim}_{\bbc}\mathrm{Hom}(Z[1], M)$ and
$$t(\alpha)=\mathrm{dim}_{k} \mathrm{Hom}(L'[1],L)\theta- \delta_{L',
M\oplus Y[1]}\mathrm{dim}_{\bbc}\mathrm{Hom}(M, Y[1]).$$Here,
$\theta$ is a morphism from $L$ to $L'[1]$ in a triangle belonging
to $\alpha$. Consider the action of $\mathrm{Aut}X$ on
$V(L',L;M\oplus X)_{Y, Z[1]}$, we always have
$s(\alpha)=s(a.\alpha)$ and $t(\alpha)=t(a.\alpha)$ for any $a\in
\mathrm{Aut} X$ and $\alpha\in V(L',L; M\oplus X)_{Y, Z[1]}.$ Hence,
the sum is again equal to
\begin{eqnarray}
   && \sum_{[L],[L'];L,L'\in\mc_2}(\sum_{\bar{\alpha}\in
\bar{V}(L',L;M\oplus X)_{Y,
Z[1]}}(\frac{1}{q^{s(\alpha)}}-\frac{1}{q^{t(\alpha)}})\cdot
\frac{|\mathrm{Aut}X|}{|G_{\alpha}|})\cdot
\frac{1}{|\mathrm{Aut}X|} \nonumber\\
   &=&\sum_{[L],[L'];L,L'\in\mc_2}(\sum_{\bar{\alpha}\in
\bar{V}(L',L;M\oplus X)_{Y,
Z[1]}}(\frac{1}{q^{s(\alpha)}}-\frac{1}{q^{t(\alpha)}})\cdot
\frac{1}{|G_{\alpha}|})\nonumber
\end{eqnarray}
where $G_{\alpha}$ is the stable subgroup for $\alpha\in
V(L',L;M\oplus X)_{Y, Z[1]}.$ If $L=0, L'\cong X\oplus Z$ or $L\cong
X\oplus Y,L'=0,$ then it is easy to know
$$\frac{|\mathrm{Hom}(M\oplus
X,L)^{Y,Z[1]}_{L'[1]}|}{|\mathrm{Aut}L|}=\frac{|\mathrm{Hom}(L',M\oplus
X)^{Y,Z[1]}_{L}|}{|\mathrm{Aut}L'|}=1.$$ Hence, if $X\cong Y[1],
X\ncong Z$, $X\ncong Z[1]$ and $M\cong Z$ (i.e., $L=0$ and $L'\cong
X\oplus Z$), then $$
[(u_Yu_X)u_Z-u_Y(u_Xu_Z)](M)=\frac{1}{|\mathrm{Aut}X|}-\frac{1}{|\mathrm{Aut}X|}\cdot
|\mathrm{Hom}(Z, X)|.
$$ By \cite[Lemma 8.1]{PX2000}, we have
$$
\frac{1}{|\mathrm{Aut}X|}-\frac{1}{|\mathrm{Aut}X|}\cdot
|\mathrm{Hom}(Z, X)|=-\frac{\mathrm{dim}_{k}\mathrm{Hom}(Z,
X)}{d(X)}.
$$
If $X\cong Z[1], X\ncong Y, X\ncong Y[1]$ and $M\cong Y$ (i.e.,
$L\cong X\oplus Y$ and $L'=0$), then
$$
[(u_Yu_X)u_Z-u_Y(u_Xu_Z)](M)=\frac{1}{|\mathrm{Aut}X|}\cdot
|\mathrm{Hom}(X, Y)|-\frac{1}{|\mathrm{Aut}X|}. $$ Otherwise,
$G_{\alpha}$ is isomorphic to a vector space by Lemma
\ref{stablesubgroup}. Therefore, the sum vanishes in
$\bbz[\frac{1}{q}]/(q-1).$ Now we assume that $L\cong X\oplus Y$ and
$L'\cong X\oplus Y[1].$ We note that Lemma \ref{commonsummand} shows
that $M\cong X$ and $Z\cong Y[1]$ in this case. By Lemma
\ref{mainproposition}, we have
$$
\frac{|\mathrm{Hom}(M\oplus
X,L)^{Y,Z[1]}_{L'[1]}|}{|\mathrm{Aut}L|}=\frac{|V(L',L; M\oplus
X)^{Y,Z[1]}|}{|\mathrm{Aut}X|\cdot|\mathrm{Hom}(Y,X)|}
$$
and
$$
\frac{|\mathrm{Hom}(L', M\oplus
X)^{Y,Z[1]}_{L}|}{|\mathrm{Aut}L'|}=\frac{|V(L',L; M\oplus
X)^{Y,Z[1]}|}{|\mathrm{Aut}X|\cdot|\mathrm{Hom}(X[1],Y)|}
$$
for $L\cong X\oplus Y$, $L'\cong X\oplus Y[1], M\cong X$ and $Z\cong
Y[1].$ In this case, we deduce that
$$
|\mathrm{Hom}(Z[1], M)|\cdot\frac{|V(L',L; M\oplus
X)_{Y,Z[1]}|}{|\mathrm{Aut}X|\cdot|\mathrm{Hom}(Y,X)|}-|\mathrm{Hom}(M,
Y[1])|\cdot\frac{|V(L',L; M\oplus
X)_{Y,Z[1]}|}{|\mathrm{Aut}X|\cdot|\mathrm{Hom}(X[1],Y)|}
$$
vanishes. This completes the proof of the theorem.
\end{proof}
We note that the theorem is a refinement of Proposition 7.5 in
\cite{PX2000} which is a crucial point to prove the Jacobi identity
in \cite{PX2000}.

For any $X\in \mathrm{ind}\mc_2$, we introduce a new variant
$\theta_{X}$ associated to $X$. Define a new multiplication between
indecomposable objects by the following rules. For $X, Y\in \mc_2$,

\begin{enumerate}
  \item  if $X\ncong Y[1],$ $u_X \cdot u_Y
  :=\sum_{[L]}g_{YX}^{L}u_{L},$
  \item $u_X\cdot u_{X[1]}:=\sum_{[L]}g_{X[1]X}^L+\theta_{X[1]},$
  \item $u_Y\cdot\theta_{X}:=-\frac{\mathrm{dim}_k\mathrm{Hom}(X[1],
Y)}{d(X)},$
  \item $\theta_{X}\cdot
u_Y:=-\frac{\mathrm{dim}_k\mathrm{Hom}(Y, X)}{d(X)}$.
\end{enumerate}

Then we have the following direct corollary of Theorem
\ref{associativity}.
\begin{Cor}\label{almostassociativity}
With the above notation,  in $\bbz[\frac{1}{q}]/(q-1)$, we have
$$
[(u_Y\cdot u_X)\cdot u_Z-u_Y\cdot (u_X\cdot u_Z)](M)=0.
$$
for $M\neq 0.$
\end{Cor}
This shows that the new multiplication is `almost' associative for
indecomposable objects. Here, `almost' means that $(u_Y\cdot
u_X)\cdot u_Z$ or $u_Y\cdot (u_X\cdot u_Z)$ may be nonsense in
$\bbz[\frac{1}{q}]/(q-1),$ $(u_Y\cdot u_X)\cdot u_Z-u_Y\cdot
(u_X\cdot u_Z)$ makes sense in $\bbz[\frac{1}{q}]/(q-1).$

\section{The theorem of Peng and Xiao}
Let $\mc_2$ be a 2-periodic triangulated category as in the last
section.  We denote by $G(\mc_2)$ the Grothendieck group of $\mc_2$.
For any $M\in\mc_2,$ we denote by $h_M$ the canonical image of $M$
in $G(\mc_2),$ called the dimension vector of $M.$ Denote by
$\bold{h}$ the subgroup of $\mathbb{Q}\otimes_{\mathbb{Z}}G(\mc_2)$
generated by $\tilde{h}_M:=\frac{h_M}{d(M)}$ for $M\in
\mbox{ind}\mc_2$. Define a symmetric bilinear function $(-\mid-)$ on
$\bold{h}\times \bold{h}$ as follows
$$
(h_X\mid
h_Y)=\mbox{dim}_{k}\mbox{Hom}(X,Y)-\mbox{dim}_{k}\mbox{Hom}(X,Y[1])+\mbox{dim}_{k}\mbox{Hom}(Y,X)-\mbox{dim}_{k}\mbox{Hom}(Y,X[1])
$$
for any $X,Y\in\mc_2.$ We note that $(\tilde{h}_X\mid h_Y)\in
\bbz[\frac{1}{q}]$ for indecomposable objects $X,Y\in \mc_2.$

Let $\mathfrak{n}$ be a free abelian group with a basis $\{[X]\mid
X\in\mbox{ind}\mc_2\}.$ Let $\bold{g}=\bold{h}\oplus \mathfrak{n}$
be a direct sum of $\mathbb{Z}$-modules. Consider the factor group
$\bold{g}_{(q-1)}=\bold{g}/(q-1)\bold{g}.$ We denote by $u_M,h_M$
the corresponding residues. As we know, the associativity of the
multiplication naturally deduces the Jacobi identity. Here, the
situation is very similar. The `almost' associativity of the
multiplication defined by $g^L_{XY}$ for indecomposable objects also
deduces the Jacobi identity. For any indecomposable objects $X,Y$ in
$\mc_2,$ we define the Lie bracket over $\bbz[\frac{1}{q}]/(q-1)$ by
$$
[u_X, u_Y] :=[u_X, u_Y]_{\mathfrak{n}}+\delta_{X,Y[1]}\tilde{h}_X.
$$
where $[u_X,
u_Y]_{\mathfrak{n}}=\sum_{[L]}(g_{YX}^{L}-g_{XY}^{L})u_{L}.$ And,
$$
[\tilde{h}_{X},u_{Y}]=-[u_Y,\tilde{h}_X]=-(\tilde{h}_X\mid h_Y)u_Y
\quad \mbox{ and } \quad [\bold{h},\bold{h}]=0.
$$
 We introduce the notation $$[u_X,
u_Y]_{\mathfrak{n}}(L):=g_{YX}^{L}-g_{XY}^{L}.$$

\begin{lemma}\cite[Lemma 7.4]{PX2000}\label{indecomposable}
With the above notation, we have
$$[u_X,
u_Y]_{\mathfrak{n}}(L)=0$$ where $L$ is a decomposable object in
$\mc_2.$
\end{lemma}
 By Lemma \ref{indecomposable}, the definition of Lie
bracket is well defined.

\begin{remark}\label{integer}
\begin{enumerate}
    \item We refer to \cite{Hubery2004} to replace
$(h_X\mid h_Y)$ in \cite{PX2000} by  $(\tilde{h}_X\mid h_Y).$
    \item In \cite{PX2000}, the Lie bracket is defined by $[u_X, u_Y] :=\sum_{[L]\in
ind\mc_2}(F_{YX}^{L}-F_{XY}^{L})u_{L}.$ However, in
$\bbz[\frac{1}{q}]/(q-1),$
$g_{YX}^{L}-g_{XY}^{L}=F_{YX}^{L}-F_{XY}^{L}.$
\end{enumerate}
\end{remark}

\begin{Prop}\label{Jacobi}
Let $X,Y,Z$ be indecomposable objects in $\mc_2.$ Then in
$\bbz[\frac{1}{q}]/(q-1)$ we have
$$
[[u_X, u_Y], u_Z]=[[u_X, u_Z], u_Y]+[u_X, [u_Y, u_Z]]
$$
\end{Prop}
\begin{proof}
 By definition, we know
$$[[u_X, u_Y], u_Z]=[[u_X, u_Y]_{\mathfrak{n}}, u_Z]_{\mathfrak{n}}-(g_{YX}^{Z[1]}-g_{XY}^{Z[1]})\tilde{h}_{Z}-\delta_{X,Y[1]}(h_X\mid h_Z)/d(X)u_Z.$$
By Lemma \ref{indecomposable}, we have
$$
[u_X, u_Y]_{\mathfrak{n}}=u_Xu_Y-u_Yu_X
$$
and
\begin{eqnarray}
   && [[u_X, u_Y]_{\mathfrak{n}}, u_Z]_{\mathfrak{n}}(M) \nonumber \\
   &=& \sum_{L\in
\mathrm{ind}\mc_2}(g_{YX}^{L}-g^{L}_{XY})[u_L, u_Z](M) \nonumber \\
   &=& \sum_{L\in
\mc_2}(g_{YX}^{L}-g^{L}_{XY})[u_L, u_Z](M) \nonumber \\
   &=&
   [(u_Xu_Y)u_Z-(u_Yu_X)u_Z-u_Z(u_Xu_Y)+u_Z(u_Yu_X)](M)\nonumber\\
   &&-g_{YX}^{M\oplus Z[1]}/|\mathrm{Aut} Z|+g_{XY}^{M\oplus
Z[1]}/|\mathrm{Aut} Z|+g_{YX}^{M\oplus Z[1]}/|\mathrm{Aut}
Z|-g_{XY}^{M\oplus Z[1]}/|\mathrm{Aut} Z|\nonumber\\
&=&
   [(u_Xu_Y)u_Z-(u_Yu_X)u_Z-u_Z(u_Xu_Y)+u_Z(u_Yu_X)](M).\nonumber
\end{eqnarray}

\nd Let $\bigtriangleup$ be the set of the permutations of
$(X,Y,Z).$ Then
$$
\bigtriangleup=\{(X,Y,Z), (Y,X,Z), (Y,Z,X), (Z,Y,X), (X,Z,Y),
(Z,X,Y)\}.
$$
Define $\bigtriangleup^{+}=\{(Y,X,Z), (Z,Y,X), (X,Z,Y))\}$ and
$\bigtriangleup^{-}=\bigtriangleup\setminus\bigtriangleup^{+}.$ If
$X, Y$ and $Z$ satisfy one of the following properties
\begin{enumerate}
    \item $ X\cong Z\cong Y[1]$,
    \item $ X\cong Y\cong Z[1]$,
    \item $ Y\cong Z\cong X[1]$,
    \end{enumerate} then it is clear that the identity in the lemma holds.
    If $X, Y, Z$ and $M$ satisfy the condition in Theorem \ref{associativity} (3),  then
using Lemma \ref{indecomposable} and Theorem \ref{associativity}, we
have
\begin{eqnarray}
   &&\{[[u_X, u_Z]_{\mathfrak{n}}, u_Y]_{\mathfrak{n}}+[u_X, [u_Y,
u_Z]_{\mathfrak{n}}]_{\mathfrak{n}}-[[u_X, u_Y]_{\mathfrak{n}},
u_Z]_{\mathfrak{n}}\}(M)\nonumber \\
   &=& \{\sum_{(A,B,C)\in
\bigtriangleup^{+}}((u_Au_B)u_C-u_A(u_Bu_C))-\sum_{(A,B,C)\in
\bigtriangleup^{-}}((u_Au_B)u_C-u_A(u_Bu_C))\}(M)\nonumber \\
    &=& \{\sum_{(A,B,C)\in
\bigtriangleup^{+}}[((u_Au_B)u_C-u_A(u_Bu_C))-((u_Cu_B)u_A-u_C(u_Bu_A))]\}(M)\nonumber \\
    &=& 0.\nonumber
\end{eqnarray}
It is enough to prove that
\begin{equation}\label{zero}
(g_{YX}^{Z[1]}-g_{XY}^{Z[1]})\tilde{h}_{Z}-(g_{ZX}^{Y[1]}-g_{XZ}^{Y[1]})\tilde{h}_{Y}-(g_{YZ}^{X[1]}-g_{ZY}^{X[1]})\tilde{h}_{X}=0.
\end{equation}
By the symmetry of the equation, we need to prove that
$$
g_{YX}^{Z[1]}\tilde{h}_{Z}+g_{XZ}^{Y[1]}\tilde{h}_{Y}+g_{ZY}^{X[1]}\tilde{h}_{X}=0.
$$
By definition, we have, over $\bbz[\frac{1}{q}]/(q-1)$,
$$\frac{g_{YX}^{Z[1]}}{d(Z)}=\frac{|(Y, Z[1])_X|}{|\mathrm{Aut} Y|\cdot d(Z)}=\frac{|(Y[1], Z)_{X[1]|}}{|\mathrm{Aut} Z|\cdot d(Y)}=\frac{g_{XZ}^{Y[1]}}{d(Y)}$$
and
$$\frac{g_{YX}^{Z[1]}}{d(Z)}=\frac{|(Z[1], X)_{Y[1]}|}{|\mathrm{Aut} X|\cdot d(Z)}=\frac{|(Z, X)_{Y|}}{|\mathrm{Aut} Z|\cdot d(X)}=\frac{g_{ZY}^{X[1]}}{d(X)}.$$
The equation \eqref{zero} follows this and the property
$h_Z+h_Y+h_X=0.$

 This shows that
the Jacobi identity holds. Now we assume that the condition of
Theorem \ref{almostassociativity} (1) holds, i.e., $X\cong Y[1]$ and
$Z\cong M$ but $X\ncong Z$ and $X\ncong Z[1].$ We note that Lemma
\ref{indecomposable} and the properness condition imply $[u_X,
u_Y]_{\mathfrak{n}}=0$ if $X\cong Y[1]$. Following Theorem
\ref{associativity} (1), we have
$$[(u_Xu_Y)u_Z-u_X(u_Yu_Z)](M)=-\frac{\mathrm{dim}_{k}\mathrm{Hom}(Z,
Y)}{d(X)}$$ and
$$[(u_Yu_X)u_Z-u_Y(u_Xu_Z)](M)=-\frac{\mathrm{dim}_{k}\mathrm{Hom}(Z,
X)}{d(X)}.$$ By Lemma \ref{indecomposable}, this implies
$$
[u_Y(u_Xu_Z)-u_X(u_Yu_Z)](M)=\frac{\mathrm{dim}_{k}\mathrm{Hom}(Z,
X)-\mathrm{dim}_{k}\mathrm{Hom}(Z, Y)}{d(X)}.
$$
Similarly, using Theorem \ref{associativity} (2), we have
$$
[(u_Zu_Y)u_X-(u_Zu_X)u_Y](M)=-\frac{\mathrm{dim}_{k}\mathrm{Hom}(X,
Z)-\mathrm{dim}_{k}\mathrm{Hom}(X, Z[1])}{d(X)}.
$$
Therefore, we obtain
$$
\{[[u_X, u_Z]_{\mathfrak{n}}, u_Y]_{\mathfrak{n}}+[u_X, [u_Y,
u_Z]_{\mathfrak{n}}]_{\mathfrak{n}}\}(Z)=-(h_X\mid h_Z)/d(X)[[u_X,
u_Y],u_Z](Z).
$$
In the same way, the Jacobi identity holds for the case that $X\cong
Z[1]$ and $Y\cong M$ but $X\ncong Y$ and $X\ncong Y[1].$ This
completes the proof of the proposition.
\end{proof}

The following is the main result  in \cite{PX2000}.
\begin{theorem}\label{maintheorem}
Endowed with the above Lie bracket, $\bold{g}_{(q-1)}$ is a Lie
algebra over $\mathbb{Z}/(q-1).$
\end{theorem}
\begin{proof}
Let $X,Y$ and $Z\in \ind\mc_2.$ By Remark 2.1, we know that in
$\bbz[\frac{1}{q}]/(q-1),$
$g_{YX}^{L}-g_{XY}^{L}=F_{YX}^{L}-F_{XY}^{L}$ and
$F_{YX}^{L}-F_{XY}^{L}\in \bbz.$ By Proposition \ref{Jacobi}, we
only need to verify the Jacobi identities
$$
[[\tilde{h}_X, u_Y], u_Z]=[[\tilde{h}_X, u_Z], u_Y]+[\tilde{h}_X,
[u_Y, u_Z]]
$$
and
$$
[[\tilde{h}_X, \tilde{h}_Y], u_Z]=[[\tilde{h}_X, u_Z],
\tilde{h}_Y]+[\tilde{h}_X, [\tilde{h}_Y, u_Z]]
$$
hold. They can be easily verified by the definition of Lie bracket.
We refer to \cite{PX2000} or \cite{Hubery2004} for details.
\end{proof}

By Corollary \ref{almostassociativity}, we can give a new version of
the above theorem.  For any $M\in\mc_2,$ we set
$${\tilde{h}}^*_M=\theta_M-\theta_{M[1]} \mbox{ and } h^*_M=d_M{\tilde{h}}^*_M.$$
It is clear that ${\tilde{h}}^*_M=-{\tilde{h}}^*_{M[1]}.$ Let
$\bold{h^*}$ be the abelian group generated by ${\tilde{h}}^*_M$ for
$M\in \mathrm{ind}\mc_2.$ The symmetric bilinear form for $\bold{h}$
naturally induces the symmetric bilinear form for $\bold{h^*}$,
denoted by $(-\mid-)^*.$ Moreover, we have the following result
analog to the additivity of dimension vectors in $G(\mc_2)$.
\begin{lemma}\label{cartanpart}
Let $M\xrightarrow{f} L\xrightarrow{g} N\xrightarrow{h} M[1]$ be a
triangle for $M, N$ and $L$ in $\mc_2.$ Then for any $Z\in \mc_2$,
we have
$$
(h^*_M+h^*_N\mid h^*_Z)^*=(h^*_L\mid h^*_Z)^*.
$$
\end{lemma}
\begin{proof}
Applying the functor $\mathrm{Hom}(Z, -)$ on the triangle, we obtain
a long exact sequence
$$
\xymatrix{\cdots \ar[r]&\mathrm{Hom}(Z,
N[-1])\ar[r]^{h^*}&\mathrm{Hom}(Z, M)\ar[r]&\cdots}
$$
Then we have
$$
\mathrm{dim}_k\mathrm{Hom}(Z, M)-\mathrm{dim}_k\mathrm{Hom}(Z,
M[1])+\mathrm{dim}_k\mathrm{Hom}(Z,
N)-\mathrm{dim}_k\mathrm{Hom}(Z,N[1])$$$$=\mathrm{dim}_k\mathrm{Hom}(Z,
L)-\mathrm{dim}_k\mathrm{Hom}(Z, L[1]).
$$
Applying the functor $\mathrm{Hom}(-, Z)$ on the triangle, we obtain
$$
\mathrm{dim}_k\mathrm{Hom}(M, Z)-\mathrm{dim}_k\mathrm{Hom}(M[1],
Z)+\mathrm{dim}_k\mathrm{Hom}(N, Z)-\mathrm{dim}_k\mathrm{Hom}(N[1],
Z)$$$$=\mathrm{dim}_k\mathrm{Hom}(L,
Z)-\mathrm{dim}_k\mathrm{Hom}(L[1], Z).
$$
This completes the proof of the proposition.
\end{proof}

\begin{lemma}\label{commutativity}For any $X, Y \mbox{ and } Z\in \mathrm{ind}\mc_2$, we have
 $$(\theta_X\cdot \theta_Y)\cdot u_Z=(\theta_Y\cdot \theta_X)\cdot u_Z\mbox{ and }u_Z\cdot(\theta_X\cdot \theta_Y) =\cdot u_Z\cdot(\theta_Y\cdot \theta_X).$$
\end{lemma}

Lemma \ref{cartanpart} and \ref{commutativity} strongly suggest us
to make the following assumptions for $\bold{h}^*.$ For any $X, Y$
and $L$ in $\mc_2$ with the triangle $X\rightarrow L\rightarrow
Y\rightarrow X[1]$, we have $h^*_X+h^*_Y=h^*_L$ and
$h^*_Xh^*_Y=h^*_Yh^*_X.$
 Let $\bold{g}^*={\bold{h}}^*\oplus
\mathfrak{n}$ and $\bold{g}^*_{(q-1)}=\bold{g}^*/(q-1)\bold{g}^*.$
Then the new multiplication in the last section  naturally induces a
Lie bracket over $\bold{g}^*_{(q-1)}$, i.e., for $X, Y\in
\mathrm{ind}\mc_2,$
$$[u_X, u_Y]^*:=u_X\cdot u_Y-u_Y\cdot u_X , \quad [{\tilde{h}}^*_X, u_Y]^*:={\tilde{h}}^*_X\cdot u_Y-u_Y\cdot {\tilde{h}}^*_X
$$
and $$[{\tilde{h}}^*_X, {\tilde{h}}^*_Y]^*:=0 .$$ Then we have an
analogue of Theorem \ref{maintheorem} which is a direct consequence
of Corollary \ref{almostassociativity} and the fact that the
`almost' associativity implies the Jacobi identity.
\begin{theorem}
With the above notation and assumptions,  $\bold{g}^*_{(q-1)}$ is a
Lie algebra over $\bbz/(q-1).$
\end{theorem}
\begin{proof}
Given indecomposable objects $X, Y$ and $Z$ in $\mc_2,$ we need to
prove the Jacobi identities
$$
[[u_X, u_Y]^{*}, u_Z]^{*}=[[u_X, u_Z]^{*}, u_Y]^{*}+[u_X, [u_Y,
u_Z]^{*}]^{*},
$$
$$
[[\tilde{h}^{*}_X, u_Y]^{*}, u_Z]^{*}=[[\tilde{h}^{*}_X, u_Z]^{*},
u_Y]^{*}+[\tilde{h}^{*}_X, [u_Y, u_Z]^{*}]^{*}
$$
and
$$
[[\tilde{h}^{*}_X, \tilde{h}_Y]^{*}, u_Z]^{*}=[[\tilde{h}^{*}_X,
u_Z]^{*}, \tilde{h}_Y]^{*}+[\tilde{h}^{*}_X, [\tilde{h}^{*}_Y,
u_Z]^{*}]^{*}
$$
hold. The second and third identities follows Proposition
\ref{cartanpart}. We prove the first identity. By Corollary
\ref{almostassociativity}, it is enough to prove
$$
[[u_X, u_Y]^{*}, u_Z]^{*}(0)=[[u_X, u_Z]^{*}, u_Y]^{*}(0)+[u_X,
[u_Y, u_Z]^{*}]^{*}(0).
$$
The proof is similar to the proof of equation \eqref{zero}.
\end{proof}
However, it is not clear whether there exists some explicit relation
between $\bold{g}^*_{(q-1)}$ and $\bold{g}_{(q-1)}$.
\\
\\
\nd {\textbf{Acknowledgements.}} Some part of this paper was written
while the author was staying at University of Bielefeld as Alexander
von Humboldt Foundation fellow. The author expresses his gratitude
to Professor C. M. Ringel for his hospitality.  The author would
like to thank the Alexander-von-Humboldt-Stiftung for a fellowship.

\bibliographystyle{amsplain}

\end{document}